\documentclass[11pt]{amsart}
\usepackage{graphicx, amsfonts, amssymb}
\usepackage{mathrsfs, amsmath, amsthm}
\usepackage{verbatim}
\usepackage{relsize}
\usepackage{epsfig}
\usepackage{color, colortbl}
\usepackage{comment}
\usepackage{mymacros}
\usepackage{xfrac,bigints}
\usepackage{enumerate}
\usepackage{hyperref}
\hypersetup{
	colorlinks=true,
	linkcolor=blue,
	filecolor=magenta,      
	urlcolor=cyan,
	pdftitle={Optimal Domain for Tg Operator},
	pdfpagemode=FullScreen,
}

\theoremstyle{definition}

 \theoremstyle{remark}

\numberwithin{equation}{section}

\def\BMOA{\text{BMOA}}

\begin{document}

\title[Toeplitz and Hankel integral operators]{On asymptotic and essential Toeplitz and Hankel integral operators}
\author{C. Bellavita}
\email{carlobellavita@ub.edu}
\address{Departament of Matem\'atica i Inform\'atica, Universitat de Barcelona, Gran Via 585, 08007 Barcelona, Spain.}

\author{G. Stylogiannis}
\email{stylog@math.auth.gr}
\address{Department of Mathematics, Aristotle University of Thessaloniki, 54124, Thessaloniki, Greece.}

\thanks{
\\
The first author is a member of Gruppo Nazionale per l’Analisi Matematica, la Probabilit\`a e le loro Applicazioni (GNAMPA) of Istituto Nazionale di Alta Matematica (INdAM) and he was partially supported by PID2021-123405NB-I00 by the Ministerio de Ciencia e Innovaci\'on and by the Departament de Recerca i Universitats, grant 2021 SGR 00087.\\
The second author was partially supported by the Hellenic Foundation for Research and Innovation (H.F.R.I.) under the '2nd Call for H.F.R.I. Research Projects to support Faculty Members \& Researchers' (Project Number: 4662).}


%Algebras of specific types of operators; ntegral operators

\begin{abstract}

In this article  we consider the integral operators 
$$
V_gf(z)=\int_{0}^{z}f(u)g'(u)\,du,\quad S_{g}f(z)=\int_{0}^{z}f'(u)g(u)\,du,
$$
acting on the Hilbert space $H^2$.
We characterize  when these operators are  uniform, strong and weakly asymptotic Toeplitz and Hankel operators. Moreover we completely describe the symbols $g$ for which these operators are essentially Hankel and essentially Toeplitz.   

%    The class of bounded operators acting on the Hardy space $H^2$ of the unit disc forms a unital, closed $\mathbb{C}$-$*$ algebra with respect to the composition operation.
%Additionally, we also consider the quotient algebra whose elements are operators modulo compact:
%$$
%[T]:=\left\lbrace T_1\in \mathcal{B}(H^2) \text{ such that } T-T_1 \text{ is a compact operator} \right\rbrace\footnote{ \text{Instead of $[T]$, we typically use the symbol $T$ to represent the equivalence class.}}.
%$$
%This quotient algebra is known as the Calkin algebra and it is denoted by $\mathcal{C}(H^2)$, see \cite{???} \CR adding reference Calkin algebra \CB. 
%This article focuses on certain closed sub-algebras of $\mathcal{B}(H^2)$ and shows that the integral operators are well-suited as their elements.
%: they satisfy the necessary conditions required for belonging to these classes.
\end{abstract}

\maketitle
%%%%%%%%%%%%%%%%%%%%%%%%%%%%%%%%%%%%%%%%%%%%%%%%%%%%%%%%%%%%%%%%%%%%%%%%%%%%%%%%%%%%%%%%%%%%%%
%%%%%%%%%%%%%%%%%%%%%%%%%%%%%%%%%%%%%%%%%%%%%%%%%%%%%%%%%%%%%%%%%%%%%%%%%%%%%%%%%%%%%%%%%%%%%%
\section{Introduction}
\noindent 
%Let $\mathcal{H}(\mathbb{D})$ be the set of all holomorphic functions in the unit disk $\mathbb{D}$. 
The underlying Hilbert space is  the Hardy space $H^2$ of the unit disk $\mathbb{D}$. We consider $\mathcal{B}(H^2)$, the algebra of bounded operators acting on $H^2$. We denote by $S \in \mathcal{B}(H^2)$ the   shift operator given by 
$$
Sf(z)=zf(z),\quad f\in H^2
$$
and by  $S^* \in \mathcal{B
}(H^2)$ the backward-shift operator defined as
$$
S^*f(z)=\frac{f(z)-f(0)}{z}\quad f\in H^2,
$$
see \cite{nikolski2002} and \cite{garcia2016} for further information.
It is clear
%\footnote{ We compose the operators from the right to the left.} 
that 
$$
S^*\, S=\text{id}_{H^2} \quad \text{ and } \quad 
S\, S^*f(z)=f(z)-f(0).
$$
For this reason, $S$ is unitary in the Calkin algebra $\mathcal{K}(H^2)$, that is 
\begin{equation}\label{Unitary S}
 S^*\,S= S\, S^* =\text{id}_{H^2} \text{ modulo compact operators}.  
\end{equation}
We denote by $\mathcal{K}(H^2)$ the sub-algebra of $\mathcal{B}(H^2)$ made by all the compact operators.

\vspace{11 pt}
\noindent 
A bounded operator $T \in \mathcal{B}(H^2)$ is a Toeplitz operator if 
$$ 
S^{*}\, T\, S - T = 0.
$$
As mentioned by Barr\'ia and Halmos in \cite{Barria1982}, right-multiplying by $S$ removes the first column of the matrix representation of $T$, while left-multiplying by $S^*$ removes the first row. Consequently, the transformation $T\mapsto S^*TS$ shifts the matrix representation of $T$ one position forward along the main diagonal. Therefore, the matrix remains unchanged under this transformation if and only if all its diagonals are constant.
It turns out that every Toeplitz operator is in the form
$$
T_{b}:=P_{+}M_{b}
$$
where $b\in L^\infty$, $M_b$ is the multiplication operator with symbol $b$ and $P_{+}$ is the orthogonal projection from $L^{2}$ onto $H^{2}$. The function $b$ is called the symbol of $T_{b}$.

\vspace{11 pt}\noindent
 The multiplication operator $M_b$, if it is expressed as a matrix operator with respect to the decomposition $L^2 = H^{2^{\perp}}\oplus H^2$, takes the form

\[
M_b=
\begin{pmatrix}
T_{\tilde{b}} & H_{b} \\
H_{\tilde{b}} & T_{b}\\
\end{pmatrix}
\]
where $\tilde{b}(z) = b(\overline{z})$, the diagonal entries are Toeplitz operators, and the antidiagonal are Hankel operators. 
If $b$ and $q$ are in $L^\infty$, then 
\begin{equation}
 T_{bq}=T_{b}T_{q} +H_{\tilde{b}} H_{q}.
\end{equation}
The significance of this equation is that the product of two Toeplitz
operators differs from a Toeplitz operator by the product of the two Hankel
operators, and every product of two Hankel operators arises in this way.
We recall that an Hankel operator $H$ is a transformation for which
 $$
 S^{*}H-H S=0.
 $$
Every Hankel operator is in the form
$$
H_{b}:=P_{-}M_{b}
$$
where $b\in L^\infty$ is not unique and $P_-=\text{id}-P_+$.

\vspace{11 pt}\noindent 
The Toeplitz algebra $\mathcal{T}$, is the uniformly closed algebra generated by the
Toeplitz operators.  It contains $\mathcal{K}(H^2)$ in a non trivial way \cite{Barria1982}. According to  \cite{Barria1982} and \cite{Feintuch1989},
every operator $T \in \mathcal{T}$ is of the form 
$T =  T_b + Q$ where $b\in L^\infty$ and $Q$
is in the commutator ideal of $\mathcal{T}$, which we denote by \textbf{$\mathcal{Q}$}. It is a long-standing problem characterizing the elements of $\mathcal{T}$ or equivalently  describing \textbf{$\mathcal{Q}$}. The classical Hilbert matrix operator, see \cite[Example 8]{Barria1982}, is a nontrivial example of a  non-Toeplitz and non-compact operator which belongs to $\mathcal{T}$. In \cite{Barria1982} Barr\'ia and Halmos raiseid the question if the classical Cesaro operator is in $\mathcal{T}$.

\vspace{11 pt}\noindent 
The $C^{*}$-algebra generated by the shift operator $S$ is of central
importance in modern analysis; it is called the Toeplitz $C^{*}$-algebra and it is denoted by $C^{*}(S)$. Every operator in $C^{*}(S)$ is of the form 
$T =  T_b + K$ where $b\in C(\mathbb{T})$ (the set of continuous functions on the unit circle $\mathbb{T}$) and $K\in\mathcal{K}(H^2)$ \cite[Theorem 4.3.2]{Arveson2002}.

\vspace{11 pt}\noindent 
Barr\'ia and Halmos called an operator $T \in \mathcal{B}(H^2)$ asymptotically Toeplitz if the
sequence of operators $\{S^{*n} TS^{n}\}_n$ converges strongly (i.e., pointwise) on $H^2$
. Being an asymptotic Toeplitz operator is a required condition for the appartenence to the Toeplitz algebra $\mathcal{T}$. Feintuch \cite{Feintuch1989} pointed out that one needs not rule out either weak or norm (i.e., uniform) operator
convergence; hence there are actually three different kinds of \emph{asymptotic toeplitzness}:
weak, strong (the original one) and uniform asymptotic toeplitzness.

\vspace{11 pt}\noindent 
Feintuch in \cite{Feintuch1989} introduced the corresponding asymptotic notions for Hankel operators. The significance of these objects is illustrated in a distance formula which is an operator theoretic analogue of results of
Adamjan, Arov and Krein \cite{Adamjan1971}. 
The set of essentially Hankel operators is introduced in \cite{MartnezAvendao2002} where some of its properties are also investigated. In particular, it is shown that this set contains some operators not of the form ‘Hankel plus compact’.

%\vspace{11 pt} 
%\noindent 
%Another important sub-algebra of $\mathcal{B}(H^2)$ is the Hankel algebra $\mathcal{T}^+$, the norm closed algebra generated by the the Toeplitz and the Hankel operators, see \cite{???} \CR Reference Hankel algebra\CB. It is clear that $\mathcal{T}\subset \mathcal{T}^+$. Moreover, according to \cite[Theorem 4 ]{Barria1982}, if the operator $T\in \mathcal{T}^+$, then it is strongly asymptotically Toeplitz.

\begin{defn}\label{D:UAT}
Let $T\in \mathcal{B}(H^2)$, we say that $T$ is:\\

\noindent $i)$ Uniformly asymptotically
Toeplitz (UAT) if $\{S^{*n} TS^{n}\}_n$ converges in the uniform operator topology.\\

\noindent $ii)$ Strongly asymptotically
Toeplitz (SAT) if $\{S^{*n} TS^{n}\}_n$ converges in the strong operator topology.\\

\noindent $iii)$ Weakly asymptotically
Toeplitz (WAT) if $\{S^{*n} TS^{n}\}_n$ converges in the weak operator topology.\\

\noindent 
It is clear that if $\{S^{*n} TS^{n}\}_n$ converges in the weak (strong, uniform) topology then its limit is a Toeplitz operator $T_b$ for some $b\in L^\infty$. We refer to the function $b$ as the symbol of $T$ and denote it as $sym(T)$. 
\end{defn}

\noindent
The SAT, UAT, WAT operators constitute subspaces which are invariant under the adjoint operation. It is clear that
$$
T \text{ is }UAT \Rightarrow T \text{ is }SAT \Rightarrow T \text{ is } WAT.
$$  
\noindent
One can also consider \emph{asymptotic hankelness}. However, since $\lim_{n \to \infty } {S^*}^nf=0$, the definition is less intuitive. We use the one provided by Feintuch in \cite{Feintuch1989} and \cite{FEINTUCH19901}.  

\vspace{11 pt}\noindent 
For $f\in H^2$ with $f(z)=\sum a_n z^n$, we define 
$$
P_{n}f(z)=\sum_{k=0}^{n}a_{k}z^{k}
$$
and on the subspace $P_n H^2$ spanned by $\{e_0, \cdots , e_n\}$, we define  the unitary operator $J_n$ as
$$
J_{n}e_{i}=e_{n-i},\quad 0\leq i\leq n.
$$
$J_n$ is usually called  the permutation operator of order $n$. We extend $J_n$ to $H^2$ by defining it to be zero on $(I - P_n) H^2$ and we denote this operator on $H^2$ by $J_n$ as well. 

\vspace{11 pt}\noindent 
For the operator $T \in \mathcal{B}(H^2)$ we consider the sequence of operators 
$$
\{H_{n}(T)\}_n= \{J_{n}\,T\, S^{n+1}\}_n.
$$

\begin{defn}\label{D:UAH}
Let $T\in \mathcal{B}(H^2)$, we say that $T$ is:\\

\noindent $i)$ Uniformly asymptotically
Hankel (UAH) if $\{H_{n}(T)\}_n$ converges in the uniform operator topology.\\

\noindent $ii)$ Strongly asymptotically
Hankel (SAH) if $\{H_{n}(T)\}_n$ converges in the strong operator topology.\\

\noindent $iii)$ Weakly asymptotically
Hankel (WAH) if $\{H_{n}(T)\}_n$ converges in the weak operator topology.\\
\end{defn}

\noindent It is known, but not immediately obvious, that when $\{H_{n}(T)\}_n$ 
converges, in the weak, strong or uniform topology, its limit is a Hankel operator $H_b$ for some $b\in L^\infty$. We will refer to the function $b$ as the symbol of $T$ and will denote it as $sym_{H}(T)$.
Obviously 
$$
T \text{ is }UAH\Rightarrow T \text{ is } SAH\Rightarrow T \text{ is } WAH.
$$
It is known \cite{Feintuch1989} that weak asymptotic toeplitzness implies weak asymptotic hankelness, and that,
in this case, the symbols are the same. Thus when $T$ is both asymptotic Toeplitz and Hankel in the uniform and strong sense, the symbols must be the same.

\vspace{11 pt}\noindent 
Another interesting sub-algebra of $\mathcal{B}(H^2)$ is the essential commutant of the shift operator.
\begin{defn}\label{D:ess Toeplitz}
An operator $T \in \mathcal{B}(H^2)$ belongs to essential commutant  $\mathcal{E}$  of the shift operator if
 $$
 ST-TS \text{ is a compact operator.}
 $$
 The operator $T$ is then called essentially Toeplitz (ess Toep).
 \end{defn}
\noindent Since $\mathcal{E}$ contains all Toeplitz operators,   Toeplitz algebra $\mathcal{T}$ is also included in $\mathcal{E}$.  It follows that $\mathcal{E}$ is closed under the formation of adjoints and is therefore a $C^{*}$-algebra. As noted in \cite{MartnezAvendao2002} one could define an asymptotic Toeplitz operator
in the Calkin algebra as an operator $T$ such that the sequence $\{S^{*n} TS^{n}\}_n$ converges in
the Calkin algebra. It turns out that this class is the same as the class of the
ess Toep operators.

\vspace{11 pt}\noindent 
Analogously one can define the subspace of essentially Hankel operators.
\begin{defn}\label{D:ess hank}
A bounded operator $T$ is essentially Hankel (ess Hank) if 
$$
S^{*}\, H-H\, S  \text{ is a compact operator}.
$$
\end{defn}
\noindent Clearly, all the Hankel operators are essentially Hankel, but ess Hank does not constitute a sub-algebra of $\mathcal{B}(H^2)$.
One could define $H$ to be an asymptotic Hankel operator in
the Calkin algebra if $S^{*n}HS^{*n}$ converges in the Calkin algebra or if $S^{n}HS^{n}$ converges
in the Calkin algebra. In either case as noticed in \cite{MartnezAvendao2002}, this class coincides with the class of the
ess Hank operators.

\vspace{11 pt}
\noindent
By extending our focus from a single element (the shift operator $S \in \mathcal{B}(H^2)$) to an entire sub-algebra $A\subset \mathcal{B}(H^2)$, we delve into the essential commutator of $A$, i.e. $\mathcal{E}(A)$.
\begin{defn}
Let $\mathcal{A}\subset\mathcal{B}(H^{2})$. The essential commutator of $\mathcal{A}$ is 
$$
\mathcal{E}(\mathcal{A})=\{T\in \mathcal{B}(H^{2}) :TX-XT\in \mathcal{K}(H^{2}), \text{ for every } X\in\mathcal{A}\}.
$$
\end{defn}

\noindent  In this article we characterize for which symbol the corresponding integral operator is SAT, UAT and WAT, and SAH, UAH and WAH. This lead us to prove that no non-compact generalized Volterra integral operator is in $C^{*}(S)$. Additionally , we describe  which integral operators are  ess Toep, ess Hank
and we prove that no non-compact integral operators belong to $\mathcal{E}(\mathcal{T})$.

\vspace{11 pt}
\noindent
The rest of the article is divided into four sections: after having recalled the definitions and the main properties of the integral operators, in Section 3 we characterize when they are UAT, SAT and WAT, and SAH, UAH and WAH. In section 4, we fix our attention on the essentially hankelness and essentially toeplitzness and we prove when the integral operators belong to these families. Finally, in section 5, we give some insights into the elements of the Toeplitz algebra.

%%%%%%%%%%%%%%%%%%%%%%%%%%%%%%%%%%%%%%%%%%%%%%%%%%%%%%%%%%%%%%%%%%%%%%%%%%%%%%%%%%%%%%%%%%%%%%
%%%%%%%%%%%%%%%%%%%%%%%%%%%%%%%%%%%%%%%%%%%%%%%%%%%%%%%%%%%%%%%%%%%%%%%%%%%%%%%%%%%%%%%%%%%%%%
\vspace{22 pt}
\section{Preliminaries on integral operators}
\noindent 
Let $f$ and $g$ be holomorphic functions in unit disk $\mathbb{D}$. The integral operators $V_g$ and $S_g$ are defined as
$$
V_gf(z)=\int_0^z f(u)g'(u)\,du\ , \quad S_gf(z)=\int_0^z f'(u)g(u)\,du.
$$ 
One of the reasons why these operators are so well studied is their connection with the multiplication operator $M_g$. Indeed
\begin{equation}\label{E:relation integral and multiplication}
  M_g=S_g+V_g+g(0)\delta_0
\end{equation}
where
$$
\delta_0f=f(0).
$$
The integral operator $V_g$  generalizes well-known operators:
if $g(z)=z$, then $V_g$ is the classical Volterra operator, usually denoted by $V$; if $g(z)=-\log(1-z)$, then $V_g$ is the shifted Cesaro operator, usually denoted by $C$.

\vspace{11 pt}\noindent
The study of the action of these integral operators on spaces of analytic functions was initiated by Pommerenke in \cite{Pommerenke1977} and continued by Aleman, Cima and Siskakis in \cite{Aleman1995} and \cite{Alemancima}. For information on $S_g$ see \cite{Anderson2010}. They obtained the following result.
\begin{thm}[Pommerenke, Aleman, Cima, Siskakis and Anderson]
 The integral operator $V_g$ is bounded on the Hardy space $H^p$ if and and only if $g$ is a bounded mean oscillation analytic function (it belongs to \text{BMOA}). The integral operator $S_g$ is bounded on the Hardy space $H^p$ if and and only if $g$ is a bounded analytic function (it belongs to $H^\infty$).
\end{thm}

\noindent
They characterized also the compact integral operators.
\begin{thm}[\label{T:compactness}Aleman, Cima, Siskakis and Anderson]
 The integral operator $V_g$ is compact on  $H^p$ if and and only if $g$ is a vanishing mean oscillation analytic function (it belongs to \text{VMOA}). On the other hand, the integral operator $S_g$ is never compact.
\end{thm}

\noindent The action of the integral operators have been widely studied, see for example  \cite{BellavitaOptimal}, \cite{Chalmoukis2021}, \cite{10.1093/imrn/rnaa070}, \cite{Pang2023} and \cite{10.1216/jie.2023.35.131}  . In this article we fix our attention on the integral operators acting on the Hardy space $H^2$. For the properties of these spaces of analytic functions we refer to the classical monographs \cite{duren2000} and \cite{garnett2006bounded}.

\vspace{11 pt}
\noindent Before delving into the computations, we highlight a property which will be often use in the rest of the article: through series expansion, for every $n \in \mathbb N$, the following identity holds
\begin{equation}\label{Commutator formula}
V_{z^n}\,V_{g}f=S^n\,V_{g}f-V_{g}\,S^nf,\quad f\in \mathcal{H}(\mathbb{D}),
\end{equation}
where $S^n$ corresponds to the application of the shift operator $n$ times.

%%%%%%%%%%%%%%%%%%%%%%%%%%%%%%%%%%%%%%%%%%%%%%%%%%%%%%%%%%%%%%%%%%%%%%%%%%%%%%%%%%%%%%%%%%%%%%
%%%%%%%%%%%%%%%%%%%%%%%%%%%%%%%%%%%%%%%%%%%%%%%%%%%%%%%%%%%%%%%%%%%%%%%%%%%%%%%%%%%%%%%%%%%%%%
\vspace{22 pt}
\section{Uniformly, strongly and weakly asymptotic toeplitzness and hankelness}
\noindent In this section, we characterize for which symbols $g$ the integral operators $S_g$ and $V_g$ are uniformly, strongly and weakly asymptotically Toeplitz and Hankel.

\vspace{11 pt}
\noindent 
 Feintuch in \cite[Theorem 4.1]{Feintuch1989} characterized the uniformly asymptotic Toeplitz operators.

\begin{thm}[Feintuch]\label{T:Feintuch}
An operator $T \in \mathcal{B}(H^2)$ is uniformly asymptotic Toeplitz if and only if 
$$
T=T_\phi+K,
$$
where $T_\phi$ is a bounded Toeplitz operator and $K\in\mathcal{K}(H^2)$.
\end{thm}
\noindent 
Thanks to Theorem \ref{T:Feintuch}, we are able to characterize which integral operators are UAT. 
\begin{thm}\label{UAT Tg and S_g}
 The integral operator $V_g$ on $H^2$ is uniformly asymptotic Toeplitz if and only if $V_g$ is compact, that is, if and only if $g \in \text{VMOA}$.\\
 The integral operator $S_g$ on $H^2$ is uniformly asymptotic Toeplitz if and only if $g\in QA:=VMOA\cap H^\infty$.
\end{thm}
\begin{proof}
We start with the $V_g$ operator. We want to show that if $V_g$ is not compact, then its difference with any Toeplitz operator fails to be compact. Since $V_g \notin \mathcal{K}(H^2)$, then $g \in \text{BMOA}\setminus \text{VMOA}$. We call
$$
\Delta_b=T_b-V_g,
$$
where $T_b$ is an arbitrary, non-zero, bounded Toeplitz operator. We want to prove that for every $b \in L^\infty(\mathbb{T})$, the operator $\Delta_b$ is never compact, which implies that $V_g $ is not UAT, because of Theorem \ref{T:Feintuch}.
We assume by contradiction that $\Delta_b$ is compact and we consider the sequence of monomials $e_{n}(z)=z^{n} \in H^2$, which is weakly converging to zero. Since $V_{g}(e_{n})\to 0$ strongly, we have, by the compactness of $\Delta_b$, that
$$
\|T_b(e_{n})\|_{H^2}\to0\quad\mbox{as } n\to\infty.
$$
On the other hand, since $\|T_b(e_{n})\|_{H^2}\to \|b\|_{L^{2}}$, the function $b$ must be the zero function. This provides the desired contradiction. If  $g \in \text{VMOA}$ by Theorem \ref{T:compactness} the operator $V_g$ is compact. Therefore by Theorem \ref{T:Feintuch} $V_g$ is UAT.

\vspace{11 pt}\noindent 
For the operator $S_g$ we reason as follows. Let $g\in \text{VMOA}\cap H^\infty$. From Theorem \ref{T:Feintuch} and the identity (\ref{E:relation integral and multiplication}),
we realize that $S_g$ is UAT. Let us now asuume that $g\in H^\infty\setminus \text{VMOA}$ and  suppose that $S_g$  is still UAT. By Theorem \ref{T:Feintuch}, $S_{g}=T_{b}+K$, where $K \in \mathcal{K}(H^2)$ and $T_b$ is a Toeplitz operator. Therefore by \ref{E:relation integral and multiplication}
$$
V_{g}=-T_{b}+M_{g}+K'
$$
with $K'\in \mathcal{K}(H^2)$ and $M_g-T_b$ another Toeplitz operator. However, because of Theorem \ref{T:Feintuch}, $V_g$ appears to be UAT, which implies that $g\in\text{VMOA}$ by the previous reasoning. This is a contradiction. The proof of the theorem is now completed.\\
\end{proof}

\noindent 
The previous Theorem  provides as an easy criterion for the membership of an integral operator in the Toeplitz  $C^*$-algebra. 
\begin{cor}
  The integral operator $V_g$ on $H^2$ is in $C^{*}(S)$ if and only if $V_g$ is compact, that is, if and only if $g \in \text{VMOA}$.\\
 The integral operator $S_g$ on $H^2$ is in $C^*(S)$ if and only if $g\in \mathcal{A}(\mathbb{D})\cap \text{VMOA}$, where $\mathcal{A}(\mathbb{D})$ denotes the disk algebra.  
\end{cor}

\vspace{11 pt}
\noindent 
Instead of considering the limit in uniform operator norm, it is possible to consider the limit in other operator topologies.
We characterize  when the integral operators are weakly asymptotically Toeplitz and strongly asymptotically Toeplitz.

\begin{thm}\label{T:WAT Tg}
The integral operator $V_g$ on $H^2$ is always weakly asymptotically Toeplitz  with weak asymptotic symbol $0$.\\ 
\noindent The integral operator $S_g$ on $H^2$ is always weakly asymptotically Toeplitz  with weak asymptotic symbol $g$.
\end{thm}
\begin{proof} 
We start with the $V_g$ operator. Let us consider $f \in H^2$. We have to check that the sequence 
$$
S^{n*}\,V_{g} \,S^{n}f(z)=\frac{1}{z^{n}}\int_{0}^{z}u^{n}f(u)g'(u)\, du
$$
converges weakly to $0$ in $H^2$. This will follow if we show that it is bounded and converges to zero point-wise.  Clearly $\{S^{n*}\,V_{g}\, S^{n}f\}_n$  is bounded on $H^2$. For every $z \in \mathbb D$, we have that
\begin{align*}
\lim_{n \to \infty}\left| S^{n*}\,V_{g}\,S^{n}f(z)\right|&= \lim_{n \to \infty} \left|\frac{1}{z^{n}}\int_{0}^{z}u^{n}f(u)g'(u)\,du\right|\\
&=\lim_{n \to \infty} \left|z\int_{0}^{1}t^{n}f(tz)g'(tz)\,dt\right|\\
&\leq \lim_{n \to \infty} \int_{0}^{1}t^{n}|f(tz)g'(tz)|\,dt\\
&\leq \lim_{n \to \infty} \frac{1}{n+1} \sup\{|f(tz)g'(tz)|: t\in[0,1]\}=0.
\end{align*}
%If $z=0$, we note that
%\begin{align*}
%\lim_{n \to \infty}\left| S^{n*}\,V_{g}\,S^{n}f(0)\right|&=0.
%\end{align*}
Consequently, $S^{n*}\,V_{g} \,S^{n}f$ converges weakly to $0$.

\vspace{11 pt}\noindent 
For $S_g$ we work as follows. Due to \eqref{E:relation integral and multiplication}, we have that
\begin{align*}
    S^*\, S_g\, S&= S^*\, M_g\, S- S^*\, V_g\, S- g(0)S^*\,\delta_0\, S=M_g-S^*\, V_g\, S.
\end{align*}
Thanks to the previous result, for every $f, h \in H^2$,
\begin{align*}
\lim_{n\to \infty}    \left\langle S^{*n}\, S_g\, S^nf,h\right\rangle &= \left\langle M_gf,h\right\rangle-\lim_{n \to \infty}\left\langle S^{*n}\, V_g\, S^nf,h\right\rangle=\left\langle M_gf,h\right\rangle.
\end{align*}
The theorem is now proved.\\
\end{proof}

\begin{thm}\label{T:SAT}
The integral operators $V_g$ and $S_g$ are always strongly asymptotically Topelitz. The strong asymptotic symbol of $V_g$ is $0$ while the strong asymptotic symbol of $S_g$ is $g$.
\end{thm}
\begin{proof}
Theorem \ref{T:WAT Tg} implies that if $V_g$ is SAT, then the asymptotic symbol should be $0$. 
For this reason, it is enough proving that for every $g\in \BMOA$ 
$$
\lim_{n\to\infty}\|S^{* n}\,V_g\,S^{n}f\|_{H^2}=0.
$$
Due to \eqref{Commutator formula}, we know that
$$
S^{* n}\,V_g\,S^{n}=S^{* n}\, S^n\, V_g-S^{* n}\, V_{z^n}\, V_g=V_g-S^{* n}\, V_{z^n}\, V_g.
$$
Consequently, in order to prove the theorem, it is enough noticing that for every $f(z)=\sum_{l}a_lz^l \in H^2$, 
$$
\lim_{n \to \infty} \|f-S^{* n}\, V_{z^n}f\|_{H^2}=0.
$$
Indeed
\begin{align*}
 f(z)-S^{* n}\, V_{z^n}f(z)=& \sum_{l}a_lz^l -\dfrac{1}{z^n}\int_{0}^z nu^{n-1}\sum_l a_l u^l\, du  \\
 =& \sum_{l}a_lz^l - \sum_l a_l \dfrac{1}{z^n} \dfrac{n}{n+l+1}z^{n+l}\\
 =& \sum_{l}a_lz^l - \sum_l a_l \dfrac{n}{n+l+1}z^{l}= \sum_{l}a_l \dfrac{l+1}{n+l+1}z^l
\end{align*}
and, by applying dominated convergence,
\begin{align*}
\lim_{n\to \infty}\| f-S^{* n}\, V_{z^n}f\|_2^2= \lim_{n\to \infty} \sum_{l}|a_l|^2 \Big|\dfrac{l+1}{n+l+1}\Big|^2=0 , 
\end{align*}
which proves the first implication.

\vspace{11 pt}\noindent
For the $S_g$ operator we work as follows.
Because of \eqref{E:relation integral and multiplication},
we have that
\begin{align*}
    S^*\, S_g\, S&= S^*\, M_g\, S- S^*\, V_g\, S- g(0)S^*\,\delta_0\, S=M_g-S^*\, V_g\, S.
\end{align*}
Due to the previous argument, we have that 
\begin{align*}
\lim_{n\to \infty}    S^{*n}\, S_g\, S^nf&= M_gf-\lim_{n \to \infty}S^{*n}\, V_g\, S^nf=M_gf.
\end{align*}
The theorem is now proved.\\
\end{proof}

\vspace{11 pt}
\noindent 
In order to characterize the integral operators which are uniformly asymptotically Hankel, we need the description provided in \cite[Theorem 4.3]{FEINTUCH19901}.
\begin{thm}[Feintuch]
   The operator $ T \in \mathcal{B}(H^2)$ is uniformly asymptotically Hankel  if and only if 
   $$
   T=T_{b}+U,
   $$
   where $b\in H^\infty+C$, where $C$ is the class of continuous functions on $\mathbb{T}$, and $U \in \mathcal{C} + \mathcal{K}(H^2)$, where  $\mathcal{C}$ denotes the (weakly closed) algebra of bounded linear operators in $H^2$ which have a lower triangular matrix representation with respect to the standard basis.
\end{thm}
\noindent Moreover, it is know that \cite[Theorem 4.2]{FEINTUCH19901}
$$
\lim_{n\to\infty}\|H_{n}(T)\|_{2}=0,\,\,\mbox{ if and only if }\,\, T\in \mathcal{C} + \mathcal{K}(H^2).
$$
The above discussion implies the following theorem. 
\begin{thm}
  The integral operators $V_{g}$ and $S_g$ are always uniformly asymptotically Hankel and the asymptotic symbol is always $0$.
\end{thm}
\begin{proof}
   
The proof is a consequence of \cite[Corollary 5.2]{FEINTUCH19901}.\\
\end{proof}

%%%%%%%%%%%%%%%%%%%%%%%%%%%%%%%%%%%%%%%%%%%%%%%%%%%%%%%%%%%%%%%%%%%%%%%%%%%%%%%%%%%%%%%%%%%%%%
%%%%%%%%%%%%%%%%%%%%%%%%%%%%%%%%%%%%%%%%%%%%%%%%%%%%%%%%%%%%%%%%%%%%%%%%%%%%%%%%%%%%%%%%%%%%%%
%\vspace{11 pt}
\section{Essential Toeplitzness and Hankelness}
\noindent 
In this section we characterize the integral operators which are ess Toep and ess Hank. 
 \begin{lem}\label{L:other defn ess Toep and Hank}
The operator $T$ is ess Hank if and only if $T-STS$ or, equivalently, $T- S^{*}TS^{*}$ is compact.
The operator $T$ is ess Toep if and only if $S\, T-T\, S$ or, equivalently, $S^{*}\, T- T\, S^{*}$ is compact. 
 \end{lem}
\begin{proof}
See \cite[p. 743]{MartnezAvendao2002}
\end{proof}
\begin{comment}
 
 \begin{proof}
Let $H$ be ess Hank, that is 
\begin{equation}\label{Ess Hank eq1}
S^*\, H-H\, S \equiv 0 \text{ mod } \mathcal{K}(H^2).    
\end{equation}

We will  use (\ref{Unitary S}) repeatedly. Multiplying (\ref{Ess Hank eq1}) from the left with $S$, we have 
$H-\,S\,H\, S \equiv0 \text{ mod } \mathcal{K}(H^2)$.
Conversely  if  $H-\,S\,H\, S \equiv0 \text{ mod } \mathcal{K}(H^2)$, multiplying by $S^*$ again from the left we have that (\ref{Ess Hank eq1}) is true.   Moreover by multiplying with $S$ from the left again  equation (\ref{Ess Hank eq1})  we have that $S^*\, H S^*-H \equiv 0 \text{ mod } \mathcal{K}(H^2)$.
Conversely  if $S^*\, H\,  S^*-H \equiv 0 \text{ mod } \mathcal{K}(H^2)$ by multiplying from the right by $S$ we have that (\ref{Ess Hank eq1}) holds. The corresponding results for $T$ follow similarly. 
%$$
%S^*\, H-H\, S= K+H-S\, H\, S=K'-\left( H-S^*\, H\, %S^*\right) 
%$$
%and
%$$
%S^*\, T\, S-T = K-\left( S\, T-T\, S\right)=K'+S^*\, T-T\, %S^*,
%$$
%where $K,K' \in \mathcal{K}(H^2)$.
 \end{proof}
\end{comment}
\noindent 
As already said, the set ess Toep is a sub-algebra of $\mathcal{B}(H^2)$, while the set ess Hank is only a norm-closed, self-adjoint, vector subspace of $\mathcal{B}(H^2)$.

\begin{thm}\label{T: ess hank}
The integral operator $V_g$ is ess Hank if and only if  $V_{g}(1-z^{2})\in \text{VMOA}$. On the other hand, the integral operator $S_{g}$ is never ess Hank.
\end{thm}
\begin{proof}
We start with the integral operator $V_g$.
We consider  
$$
V_{g}f(z)-S\, V_{g}\, Sf(z)=\int_{0}^{z}f(u)g'(u)\,du-
z\int_{0}^{z}f(u)ug'(u)\,du.
$$
By applying  the  formula (\ref{Commutator formula}) to $S\, V_g$, we have that 
$$
S\, V_{g}\, Sf=V_{g}\, S^2f+V_z\, V_{g}\, Sf.
$$ 
The operator $V_z\, V_{g}\, Sf$ is compact. Therefore, we have 
\begin{align*}
 V_{g}f(z)-S\, V_{g}\, S f(z)&\equiv V_{g}(f-S^2f)(z)\text{ mod } \mathcal{K}(H^2)\\
 &\equiv\int_{0}^{z}f(u)(1-u^{2})g'(u)\,du\text{ mod } \mathcal{K}(H^2)
\end{align*}
which proves the first part of the theorem.\\

\vspace{11 pt}\noindent
Let us now consider the operator $S_g$. Let us assume that $g(z)=\sum_{n=0}^{\infty} a_{n}z^{n}$ and, by contradiction, that $S_{g}$ is ess Hank.
Let $a_{k_{0}}$ be the first non-zero Taylor coefficient. Let us consider $f_{n}(z)=z^{n}, n>1$. Then $f_{n}\to 0$ weakly in $H^2$ as $n\to \infty$. Thus 
\begin{equation}\label{norm to 0 Sg}
    \|(S_{g}-S\, S_{g}\, S)f_{n}\|_{H^2}\to 0 \text{ as } n\to \infty. 
\end{equation}
However, 
\begin{align*}
  (S_{g}-S\, S_{g}\, S)f_{n}(z)&=\int_{0}^{z}g(u)(nu^{n-1})\, du -z\int_{0}^{z} g(u)((n+1)u^{n})du\\
  &=n\int_{0}^{z}\sum_{k=0}^{\infty}a_{k}u^{k+n-1}\,du-z(n+1)\int_{0}^{z}\sum_{k=0}^{\infty}a_{k}u^{k+n}\,du\\
  &=\sum_{k=0}^{\infty}\frac{n}{k+n} a_{k} z^{k+n}-\sum_{k=0}^{\infty}\frac{n+1}{k+n+1} a_{k} z^{k+n+2}\\
  &=\frac{n}{n+k_0}a_{k_{0}}z^{k_{0}+n}+\text{higher  terms.}
%  &=\sum_{l=n}^{\infty}\frac{n}{l} a_{l-n} z^{l}-\sum_{j=n+2}^{\infty}\frac{n+1}{j+1} a_{j-n-2} z^{j}\\
%  &= \dfrac{n}{k_0+n}a_{k_0}z^{k_0+n}+\frac{n+1}{k_0+n+2}a_{k_0+1}z^{k_0+n}\\
%  &\quad +\sum_{j=k_0+n+1}^{\infty}\left(\frac{n}{j+1} a_{j-n+1}-\frac{n+1}{j}a_{j-n-1} \right)z^{j}.
  \end{align*}
Consequently
$$\lim_{n\to\infty}\|(S_{g}-S\, S_{g}\, S)f_{n}\|_{H^2}\geq |a_{k_0}|$$
which contradicts (\ref{norm to 0 Sg}).
The theorem is now proved.
\end{proof}

\noindent From Theorem \ref{T: ess hank} it is clear that the Cesaro operator, that is, $V_{-\log(1-z)}$ is essentially Hankel since
$$
V_{-\log(1-\cdot)}(1-z^2)=\int_0^z \left(1+\xi\right)d\xi \in \text{VMOA}.
$$
Moreover, even if $V_{-\log(\alpha-z)}$ with $\alpha \in \mathbb{T}$ is unitarly equivalent to $V_{-\log(1-z)}$,  the operator $V_{-\log(\alpha-z)}$ is  essentially Hankel only for $a=\pm1$.

\begin{thm}\label{T:Ess Toep}
The integral operators $V_g$ and $S_g$ are always ess Toep.
\end{thm}
\begin{proof}
Again we start with $V_g$. By using the  formula \eqref{Commutator formula}, we have that
\begin{align*}
S\, V_{g}f-V_{g}\, Sf&=V_{g}\,Sf+V_z\, V_{g}f-V_{g}\, Sf=V_z\,V_{g}f.
\end{align*}
The last operator is compact, which finishes the first part of the proof of the theorem.

\vspace{11 pt}\noindent 
We consider now the $S_g$ operator. Through the identity 
$$
V_{g}+S_{g}-M_g+g(0)\,\delta_{0}=0,
$$
we  have
\begin{align*}
0&=S(V_{g}+S_{g}-M_g+g(0)\,\delta_{0})-(V_{g}+S_{g}-M_g+g(0)\,\delta_{0})S\\
&=SV_{g}-V_{g}S+S S_{g}-S_{g}S +g(0)S\delta_{0}\\
&\equiv S S_{g}-S_{g}S \text{ mod } \mathcal{K}(H^2).
\end{align*}
The theorem is now proved.\\
\end{proof}

\noindent It easy to prove that $V_g \not\in \mathcal{E}(\mathcal{T})$, when $V_g$ is not compact.
\begin{cor}
Let $g\in BMOA\setminus VMOA$. Then $V_g\not\in \mathcal{E}(\mathcal{T})$.
\end{cor}
\begin{proof}
We can actually prove that $V_g \not\in \mathcal{E}(\mathcal{T}_a) $, where $\mathcal{T}_a$ is the set made by analytic Toeplitz operators. Indeed, since Theorem \ref{D:UAT} proves that $V_g$ is not the sum of a Toeplitz and a compact operator, an application of \cite[Theorem 2]{Davidson1977} shows that there exists always a function $b\in H^\infty$ such that $M_bV_g-V_{g}M_b$ is non-compact. The proof is now completed.\\
\end{proof}

\noindent We highlight that ess Hank has some kind of ideal structure.
\begin{cor}\label{Ideal}
Let $V_g$ and $V_h$ be ess Hank and $V_{k}$ be  ess Toep. Then $V_g\,V_k$ and $V_k\,V_g$ are ess Hank. On the other hand, $V_g\, V_h$ is  ess Toep. 
\end{cor}
\begin{proof}
This corollary follows from \cite[Lemma 2.4]{MartnezAvendao2002}. However for the sake of completeness, we add here the proof. Indeed
\begin{align*}
    V_g\, V_k -S\, V_g\, V_k\, S\equiv& V_g\, V_k -S\, V_g\, S\, V_k\text{ mod } \mathcal{K}(H^2)\\
    \equiv&\left( V_g-S\, V_g\, S\right)\, V_k \text{ mod } \mathcal{K}(H^2)\\
    \equiv& 0
\end{align*}
 Similar computations work also for $V_k\,V_g $. Finally
\begin{align*}
    S\, V_g\, V_h -V_g\, V_h\, S=& S\, V_g\,S\, S^*\,  V_h - V_g\, V_h\, S \text{ mod } \mathcal{K}(H^2).\\
    \equiv&V_g S^*\,  V_h\, S^* S - V_g\, V_h\, S \text{ mod } \mathcal{K}(H^2)\\
    \equiv& V_g\, V_h\, S-V_g\, V_h\, \text{ mod } \mathcal{K}(H^2)\\
    \equiv&0 \text{ mod } \mathcal{K}(H^2).
\end{align*}

\end{proof}

%%
%%%%%%%%%%%%%%%%%%%%%%%%%%%%%%%%%%%%%%%%%%%%%%%%%%%%%%%%%%%%%%%%%%%%%%%%%%%%%%%%%%%%%%%%%%%%%%%%%%%%%%%%%%%%%%%%%%%%%%%%%%%%%%%%%%%%%%%%%%%%%%%%%%%%%%%%%%%%%%%%%%%%%%%%%%%%%%%%%%%%%%%%%%%%%%%%%%%%%%%%%%%%%%%%%%%%%%%%%%%%%%%%%%%%%%%
%%
%%
%\vspace{11 pt}
\section{remark on sub-algebra of $\mathcal{B}(H^2)$}
\noindent
The explicit characterization of elements within the Toeplitz algebra  $\mathcal{T}$ remains an open problem. One of the aim of this paper is  shedding some light on this issue by exploring the potential of integral operators  as non-trivial members of  $\mathcal{T}$. 

\vspace{11 pt}\noindent

We highlight that thanks to Theorem \ref{T:Feintuch}, if an operator $T$ is uniformly asymptotically Toeplitz, then it also needs to be essentially Toeplitz. However this inclusion is strict, due to Theorems \ref{UAT Tg and S_g} and \ref{T:Ess Toep}. Indeed, if $g \in \text{BMOA}\setminus \text{VMOA}$, the operator $V_g$ is ess Toep but not UAT.  The preceding comment are evidence, however weak, that every non-compact $V_g$ does not belong
 to $\mathcal{T}$, see also \cite[Example 10]{Barria1982} for the corresponding talk about the C\'{e}saro operator.

\vspace{11 pt}\noindent 
Notably, the linear closure of all positive Hankel operators on $H^2$ and the generalized Hilbert operators $H_\mu$ defined and studied in \cite{Gala} and \cite{CHATZIFOUNTAS2014154} belong to $\mathcal{T}$. The proof is heavily  based on  \cite[Example 8]{Barria1982}.

\begin{prop}
Every positive bounded Hankel operator is in $\mathcal{T}$. In particular 
$$
H_{\mu}f(z)=\int_{0}^{1}\dfrac{f(t)}{1-tz}d\mu(t),
$$
with $\mu$ a positive measure on $[0,1)$, belongs to $\mathcal{T}$ when it is bounded.
\end{prop}

\vspace{22 pt}\noindent 
\bibliographystyle{plain}

\bibliography{ProJectA_B}
\end{document}